\documentclass[12pt]{article}%
\usepackage{amsmath}
\usepackage{amsfonts}
\usepackage{amssymb}
\usepackage{graphicx}%
\providecommand{\U}[1]{\protect\rule{.1in}{.1in}}

\newtheorem{theorem}{Theorem}
\newtheorem{acknowledgement}[theorem]{Acknowledgement}

\newtheorem{example}[theorem]{Example}

\newtheorem{remark}[theorem]{Remark}

\newenvironment{proof}[1][Proof]{\noindent\textbf{#1.} }{\ \rule{0.5em}{0.5em}}

\begin{document}

\title{Polynomial solutions of differential-difference equations}
\author{Diego Dominici\\Department of Mathematics\\State University of New York at New Paltz\\1 Hawk Dr. Suite 9\\New Paltz, NY 12561-2443\\USA\\dominicd@newpaltz.edu.
\and Kathy Driver\\Department of Mathematics and Applied Mathematics\\University of Cape Town\\Private Bag X3, Rondebosch 7701\\Cape Town\\South Africa\\kathy.driver@uct.ac.za.
\and Kerstin Jordaan\\Department of Mathematics and Applied Mathematics\\University of Pretoria\\Pretoria, 0002\\South Africa\\kjordaan@up.ac.za}
\maketitle

\begin{abstract}
We investigate the zeros of polynomial solutions to the
differential-difference equation
\[
P_{n+1}(x)=A_{n}(x)P_{n}^{\prime}(x)+B_{n}(x)P_{n}(x),~ n=0,1,\dots
\]
where $A_{n}$ and $B_{n}$ are polynomials of degree at most $2$ and $1$
respectively. We address the question of when the zeros are real and simple
and whether the zeros of polynomials of adjacent degree are interlacing. Our
result holds for general classes of polynomials but includes sequences of
classical orthogonal polynomials as well as Euler-Frobenius, Bell and other polynomials.

\end{abstract}

\noindent AMS MOS Classification: 33C45, 42C05\quad\medskip


\section{Introduction}

Let $\{P_{n}(x)\}_{n=0}^{\infty}$ be the sequence of polynomials defined by
\begin{align}
P_{0}(x)  &  =1\nonumber\\
P_{n+1}(x)  &  =A_{n}(x)P_{n}^{\prime}(x)+B_{n}(x)P_{n}(x),~n=0,1,\dots,
\label{dde}%
\end{align}
where $A_{n}(x)$ and $B_{n}(x)$ are polynomials in $x$ of degree at most $2$
and $1$ respectively. The best-known families of classical orthogonal
polynomials satisfy differential-difference equations of this type:

\begin{itemize}
\item Jacobi polynomials
\begin{align*}
P_{n+1}^{\left(  \alpha,\beta\right)  }(x)  &  =\frac{\left(  2n+2+\alpha
+\beta\right)  \left(  x^{2}-1\right)  }{2(n+1)\left(  n+1+\alpha
+\beta\right)  }\frac{d}{dx}P_{n}^{\left(  \alpha,\beta\right)  }(x)\\
&  +\frac{\left(  2n+2+\alpha+\beta\right)  x+\alpha-\beta}{2(n+1)}%
P_{n}^{\left(  \alpha,\beta\right)  }(x).
\end{align*}

\item Laguerre polynomials%
\[
L_{n+1}^{\left(  \alpha\right)  }(x)=\frac{x}{n+1}\frac{d}{dx}L_{n}^{\left(
\alpha\right)  }(x)+\frac{\alpha+n+1-x}{n+1}L_{n}^{\left(  \alpha\right)
}(x).
\]

\item Hermite polynomials%
\[
H_{n+1}(x)=-\frac{d}{dx}H_{n}(x)+2xH_{n}(x).
\]

\end{itemize}

In these cases, it follows from the theory of orthogonal polynomials
\cite{MR0372517} that the zeros of $P_{n}(x)$ are real and simple and that the
zeros of $P_{n+1}(x)$ and $P_{n}(x)$ are interlacing. This leads to the
question of what can one say about the zeros of a sequence of polynomials
satisfying (\ref{dde}) which is not orthogonal. Examples of such sequences
include the Bell polynomials \cite{MR0146097}, the Euler-Frobenius polynomials
\cite{frobenius} and the so-called derivative polynomials \cite{MR1321452}.

\medskip A number of authors have investigated the properties of the zeros of
sequences of polynomials that are solutions of (\ref{dde}) but are not, in
general, orthogonal. In \cite{MR1174072}, Vertgeim considered polynomials
generated by (\ref{dde}) with
\[
A_{n}(x)=a_{n}x^{2}-b_{n},\quad B_{n}(x)=\alpha a_{n}x,\quad\alpha,a_{n}%
,b_{n}>0,
\]
that generalize the Euler polynomials. In \cite{MR1471728}, \cite{MR1359933},
Dubeau and Savoie study interlacing properties of the zeros of polynomial
solutions of (\ref{dde}) with
\[
A_{n}(x)=\kappa_{n}\left(  1-x^{2}\right)  ,\quad B_{n}(x)=-2\kappa_{n}%
r_{n}x,\quad r_{n}>0,\kappa_{n}\neq0,
\]
which contain the generalized Euler-Frobenius and the ultraspherical
polynomials as special cases. They also consider the Hermite-like polynomials
defined by (\ref{dde}) with
\[
A_{n}(x)=\kappa_{n},\quad B_{n}(x)=-2\kappa_{n}x,\quad\kappa_{n}\neq0.
\]
In \cite{MR2311051}, Liu and Wang analyze polynomial solutions of the equation%
\begin{equation}
P_{n+1}(x)=A_{n}(x)P_{n}^{\prime}(x)+B_{n}(x)P_{n}(x)+C_{n}(x)P_{n-1}%
(x),~n=0,1,\ldots. \label{dde1}%
\end{equation}
By assuming that $P_{n}(x)$ has strictly nonnegative coefficients
(respectively alternating in sign) and $A_{n}(x)<0$ or $C_{n}(x)<0$ for
$x\leq0$ (respectively $x\geq0),$ , they show that $\{P_{n}(x)\}_{n=0}%
^{\infty}$ forms a Sturm sequence. Although (\ref{dde1}) is more general than
(\ref{dde}), the conditions on the sign of the coefficients of $P_{n}(x)$ is a
priori very difficult to check, except in rather isolated situations.

\medskip In this paper, we shall take an approach similar to that used in
\cite{MR1174072}. We establish criteria that ensure, for a sequence of
polynomials $\{P_{n}(x)\}_{n=0}^{\infty}$ satisfying (\ref{dde}), either that
all zeros of $P_{n}$ are real and simple, or the zeros of $P_{n}$ and
$P_{n+1}$ are interlacing, or both, based on conditions that can be checked
directly from $A_{n}$ and $B_{n}$. We present several interesting examples and
consider possible extensions.

\section{Preliminary results}

It is obvious that any sequence of polynomials is a solution of (\ref{dde}),
if we allow $A_{n}(x)$ and $B_{n}(x)$ to be (non unique) rational functions of
$x.$ The question of characterizing which sequences of polynomials
$\{P_{n}(x)\}_{n=0}^{\infty}$ are solutions of (\ref{dde}), when $A_{n}(x)$
and $B_{n}(x)$ are polynomials in $x$, is addressed in the following theorem.

\begin{theorem}
\label{AB}Suppose that $P_{n}(x)$ is a polynomial with $n$ simple zeros. Let
\begin{equation}
P_{n}(x_{n,k})=0,\quad y_{n,k}=\frac{P_{n+1}(x_{n,k})}{P_{n}^{\prime}%
(x_{n,k})},\quad k=1,\dots,n. \label{xkyk}%
\end{equation}
Then the following are equivalent:

\begin{description}
\item[i)] There exist polynomials $A_{n}(x)$ and $B_{n}(x)$ of degree at most
$2$ and $1$ respectively such that $\{P_{n}(x)\}_{n=0}^{\infty}$ satisfies
(\ref{dde}). For every $n\geq3,$ $A_{n}(x)$ and $B_{n}(x)$ are unique.

\item[ii)] For every $n\geq2$, there exists a quadratic polynomial $A_{n}(x)$
such that%
\[
A_{n}(x_{n,k})=y_{n,k},\quad k=1,\dots,n.
\]

\end{description}
\end{theorem}

\begin{proof}
Evaluating (\ref{dde}) at the zeros of $P_{n}(x)$, we clearly see that i)
$\Longrightarrow$ ii). Assume now that ii) holds. For $n=0,$ we have%
\[
P_{1}(x)=A_{0}(x)P_{0}^{\prime}(x)+B_{0}(x)P_{0}(x)=B_{0}(x),
\]
since $P_{0}(x)=1.$ Thus, $A_{0}(x)$ can be any polynomial of degree at most 2
and $B_{0}(x)=P_{1}(x).$ For $n=1,$ let $B_{1}(x)$ be an arbitrary polynomial
of degree at most $1$ and define $A_{1}(x)$ by
\[
A_{1}(x)=\frac{P_{2}(x)-B_{1}(x)P_{1}(x)}{P_{1}^{\prime}(x)}.
\]
Since $P_{1}^{\prime}(x)$ is a non-zero constant, we see that $i)$ holds for
$n=1.$ For $n\geq2,$ let $A_{n}(x)$ be a quadratic polynomial such that%
\[
y_{n,k}=A_{n}(x_{n,k}),\quad k=1,\dots,n.
\]
Then $P_{n+1}(x)-A_{n}(x)P_{n}^{\prime}(x)$ is a polynomial of degree at most
$n+1$ that is zero at each $x_{n,k}$ and therefore divisible by $P_{n}(x)$.
This yields the existence of a $B_{n}$ that is at most linear. When $n\geq3$,
$A_{n}$ is uniquely determined so that $i)$ holds.
\end{proof}

\begin{remark}
The assumption that $P_{n}(x)$ has n simple zeros is not very restrictive in
the sense that the statement and proof of Theorem 1 can be modified to cater
for this possibility.
\end{remark}

\begin{example}
Let $P_{n}(x)$ be the family of orthonormal polynomials with respect to the
Freud-type weight%
\[
w(x)=\sqrt{\frac{2}{t}}\frac{1}{\mathrm{K}_{\frac{1}{4}}\left(  \frac{t^{2}%
}{2}\right)  }\exp\left(  -x^{4}+2tx^{2}-\frac{t^{2}}{4}\right)  ,
\]
where $\mathrm{K}_{\nu}\left(  z\right)  $ is the Bessel function of the
second kind. In this case, we have \cite{MR2191786}
\[
A_{n}(x)=-\frac{1}{4a_{n+1}\left(  x^{2}+a_{n+1}^{2}+a_{n}^{2}-t\right)
},\quad B_{n}(x)=\frac{x\left(  x^{2}+a_{n+1}^{2}-t\right)  }{a_{n+1}\left(
x^{2}+a_{n+1}^{2}+a_{n}^{2}-t\right)  },
\]
where the numbers $a_{n}$ are the coefficients in the three-term recurrence
relation%
\[
xP_{n}(x)=a_{n+1}P_{n+1}(x)+a_{n}P_{n-1}(x)
\]
satisfying the string equation%
\[
n=4a_{n}^{2}\left(  a_{n+1}^{2}+a_{n}^{2}+a_{n-1}^{2}-t\right)  ,\quad
n\geq0.
\]
The initial values for $a_{n}$ are%
\begin{equation}
a_{0}=0,\quad a_{1}(t)=\sqrt{\frac{t}{2}}\sqrt{\frac{\mathrm{K}_{\frac{1}{4}%
}\left(  \frac{t^{2}}{4}\right)  }{\mathrm{K}_{\frac{3}{4}}\left(  \frac
{t^{2}}{4}\right)  }-1},\quad a_{1}(0)=\frac{\Gamma\left(  \frac{3}{4}\right)
}{2^{\frac{1}{4}}\sqrt{\pi}}, \label{a0a1}%
\end{equation}
where $a_{1}>0$ is chosen so that $P_{1}(x)=\frac{x}{a_{1}}$ has unit norm.

For $n=5,$ we have%
\[
P_{5}(x)=\frac{x^{5}-\alpha x^{3}+\beta x}{a_{1}a_{2}a_{3}a_{4}a_{5}},
\]
with%
\[
\alpha=a_{1}^{2}+a_{2}^{2}+a_{3}^{2}+a_{4}^{2},\quad\beta=a_{1}^{2}a_{3}%
^{2}+a_{1}^{2}a_{4}^{2}+a_{2}^{2}a_{4}^{2}%
\]
and therefore, with the notation of (\ref{xkyk}),%
\[
x_{5,1}=0,\quad x_{5,2}=\sqrt{\zeta_{+}},\quad x_{5,3}=-\sqrt{\zeta_{+}},\quad
x_{5,4}=\sqrt{\zeta_{-}},\quad x_{5,5}=-\sqrt{\zeta_{-}}%
\]%
\[
\zeta_{\pm}=\frac{1}{2}\left(  \alpha\pm\sqrt{\alpha^{2}-4\beta}\right)  .
\]
One can show, using a simple algebraic argument, that the polynomial
interpolating the points $\left(  x_{5,k},y_{5,k}\right)  ,$ $1\leq k\leq5$
has degree $4,$ unless $a_{1}=0,$ which contradicts (\ref{a0a1}). We deduce
from Theorem \ref{AB} that there are no polynomials $A_{n}(x)and B_{n}(x)$
such that $P_{n}(x)$ satisfies (\ref{dde}).
\end{example}

\section{Main result}

In our proofs, we will find it convenient to re-write the
differential-difference equation (\ref{dde}) in the form
\begin{equation}
P_{n+1}(x)=\frac{A_{n}(x)}{K_{n}(x)}\frac{d}{dx}[K_{n}(x)P_{n}(x)],
\label{int}%
\end{equation}
where $K_{n}(x)$ is an "integrating factor", defined by
\begin{equation}
{K_{n}(x)=\exp}\left(  \int\limits^{x}\frac{B_{n}(t)}{A_{n}(t)}~dt\right)  .
\label{Kn}%
\end{equation}
Since $A_{n}(x)$ and $B_{n}(x)$ are polynomials, we can obtain all possible
functions ${K_{n}(x)}$ by considering the location of the zeros of $A_{n}(x)$
and $B_{n}(x).$ This leads to the following classification, where we define
the extended real line by $\left(  -\infty,\infty\right)  \cup\left\{
-\infty,\infty\right\}  $ and the notation $f(\pm\infty)=0$ means
$\underset{x\rightarrow\pm\infty}{\lim}f(x)=0.$

\begin{theorem}
Let ${K_{n}(x)}$ be defined by (\ref{Kn}). Then, ${K_{n}(x)}$ has at most $2$
zeros on the extended real line.
\end{theorem}

\begin{proof}
Since ${K_{n}(x)}$ depends on the ratio $B_{n}(x)/A_{n}(x),$ without loss of
generality, we can choose $A_{n}(x)$ to be monic. We consider the possible
cases in turn.

\begin{enumerate}
\item Let $A_{n}(x)=\left(  x-\lambda_{n}\right)  \left(  x-\xi_{n}\right)  ,$
$\deg\left(  B_{0}\right)  =1$ and $B_{n}(x)=\mu_{n},$ $\mu_{n}\neq0$ for
$n\geq1.$ If $\lambda_{n}\neq\xi_{n}$, we have%
\[
{K_{n}(x)=\exp}\left[  \frac{\mu_{n}}{\lambda_{n}-\xi_{n}}\ln\left(
\frac{x-\lambda_{n}}{x-\xi_{n}}\right)  \right]  .
\]
If $\xi_{n}=\lambda_{n}$, we have%
\[
{K_{n}(x)=\exp}\left(  -\frac{\mu_{n}}{x-\lambda_{n}}\right)  .
\]

\item Let $A_{n}(x)=\left(  x-\lambda_{n}\right)  \left(  x-\xi_{n}\right)  ,$
and $B_{n}(x)=\kappa_{n}\left(  x-\mu_{n}\right)  ,$ $\kappa_{n}\neq0.$ If
$\xi_{n}\neq\lambda_{n}\neq\mu_{n}$, we have%
\[
{K_{n}(x)=\exp}\left[  \kappa_{n}\frac{\lambda_{n}-\mu_{n}}{\lambda_{n}%
-\xi_{n}}\ln\left(  x-\lambda_{n}\right)  +\kappa_{n}\frac{\mu_{n}-\xi_{n}%
}{\lambda_{n}-\xi_{n}}\ln\left(  x-\xi_{n}\right)  \right]  .
\]
If $\xi_{n}=\lambda_{n}\neq\mu_{n}$, we have%
\[
{K_{n}(x)=\left(  x-\lambda_{n}\right)  }^{\kappa_{n}}{\exp}\left(  \kappa
_{n}\frac{\mu_{n}-\lambda_{n}}{x-\lambda_{n}}\right)  .
\]
If $\xi_{n}\neq\lambda_{n}=\mu_{n}$ or $\xi_{n}=\lambda_{n}=\mu_{n},$ we have%
\[
{K_{n}(x)=}\left(  {x-\xi_{n}}\right)  ^{\kappa_{n}}.
\]

\item Let $A_{n}(x)=x-\lambda_{n},$ and $B_{n}(x)=\kappa_{n}\left(  x-\mu
_{n}\right)  ,$ $\kappa_{n}\neq0.$ If $\lambda_{n}\neq\mu_{n}$, we have%
\[
{K_{n}(x)=\exp}\left[  \kappa_{n}x+\kappa_{n}\left(  \lambda_{n}-\mu
_{n}\right)  \ln\left(  x-\lambda_{n}\right)  \right]  .
\]
If $\lambda_{n}=\mu_{n}$, we have%
\[
{K_{n}(x)=\exp}\left(  \kappa_{n}x\right)  .
\]

\item Let $A_{n}(x)=1,$ and $B_{n}(x)=\kappa_{n}\left(  x-\mu_{n}\right)  ,$
$\kappa_{n}\neq0.$ We have%
\[
{K_{n}(x)=\exp}\left[  \kappa_{n}x\left(  \frac{x}{2}-\mu_{n}\right)  \right]
.
\]

\end{enumerate}
\end{proof}

We are know ready to prove our main result.

\begin{theorem}
\label{main} We denote the zeros of $P_{n}(x)$ in increasing order by
$\gamma_{i,n}$, $i=1,\dots,n$. Let
\[
-\infty\leq\alpha_{n+1}\leq\alpha_{n}<\beta_{n}\leq\beta_{n+1}\leq\infty
\]
and ${K_{n}(x)}$ be continuous on $\left[  \alpha_{n},\beta_{n}\right]  $ and
differentiable on $\left(  \alpha_{n},\beta_{n}\right)  .$Then,

\begin{itemize}
\item[(a)] If $K_{n}(x)=0$ only at $x=\alpha_{n}$ and $x=\beta_{n}$ for $1\leq
n\leq N$ and $\alpha_{1}<\gamma_{1,1}<\beta_{1}$, then the zeros of $P_{n}$
and $P_{n+1}$ interlace and are in the interval $(\alpha_{n},\beta_{n})$ for
$1\leq n\leq N$.

\item[(b)] If $K_{n}(x)=0$ only when $x=\alpha_{n}$ (respectively $\beta_{n}%
$), ${\frac{A_{n}(x)}{K_{n}(x)}=0}$ when $x=\beta_{n}$ (respectively
$\alpha_{n}$), $\beta_{n}<\beta_{n+1}$ for $1\leq n\leq N$ (respectively
$\alpha_{n+1}<\alpha_{n}$ for $1\leq n\leq N)$ and $\alpha_{1}<\gamma
_{1,1}<\beta_{1}$, then the zeros of $P_{n}$ and $P_{n+1}$ interlace and are
in the interval $(\alpha_{n},\beta_{n})$ for $1\leq n\leq N.$

\item[(c)] If $K_{n}(x)=0$ only when $x=\alpha_{n}$ (respectively $\beta_{n}%
$), ${\frac{A_{n}(x)}{K_{n}(x)}=0}$ when $x=\beta_{n}$ (respectively
$\alpha_{n}$) for $1\leq n\leq N$ and $\alpha_{1}<\gamma_{1,1}\leq\beta_{1}$
(respectively $\alpha_{1}\leq\gamma_{1,1}<\beta_{1}$), then all the zeros of
$P_{n}$ are real, simple and in the interval $(\alpha_{n},\beta_{n}]$
(respectively $[\alpha_{n},\beta_{n})$ ) for $1\leq n\leq N$.

\item[(d)] If $K_{n}(x)\neq0$, ${\frac{A_{n}(x)}{K_{n}(x)}}=0$ when
$x=\alpha_{n},\beta_{n}$ and $\alpha_{n+1}<\alpha_{n}<\beta_{n}<\beta_{n+1}$
for $1\leq n\leq N$ with $\alpha_{1}<\gamma_{1,1}<\beta_{1}$, then the zeros
of $P_{n}$ and $P_{n+1}$ interlace for $1\leq n\leq N$.
\end{itemize}
\end{theorem}

\begin{proof}
We use the familiar extension of Rolle's theorem to an infinite open interval.

\begin{itemize}
\item[(a)] When $n=1$ we have $P_{2}(x)=\frac{A_{1}(x)}{K_{1}(x)}\frac{d}%
{dx}[K_{1}(x)P_{1}(x)]$ and $K_{1}(x)P_{1}(x)=0$ at $\alpha_{1}$, $\beta_{1}$
and $\gamma_{1,1}$ with $\alpha_{1}<\gamma_{1,1}<\beta$. It follows from
Rolle's theorem that $P_{2}(x)=0$ at $\gamma_{1,2}$, $\gamma_{2,2}$ with%
\[
\alpha_{2}\leq\alpha_{1}<\gamma_{1,2}<\gamma_{1,1}<\gamma_{2,2}<\beta_{1}%
\leq\beta_{2}.
\]
Now, let $2\leq n\leq N$ and assume that we have proved the result for
$P_{n}(x).$ Since $K_{n}(x)P_{n}(x)$ vanishes at $\alpha_{n}$, $\beta_{n}$ and
$\gamma_{i,n}$ for $i=1,\dots,n$ with $\alpha_{n}<\gamma_{1,n}<\ldots
<\gamma_{n,n}<\beta_{n}$, Rolle's theorem applied to (\ref{int}) yields
\[
\alpha_{n+1}\leq\alpha_{n}<\gamma_{1,n+1}<\gamma_{1,n}<\dots<\gamma
_{n,n}<\gamma_{n+1,n+1}<\beta_{n}\leq\beta_{n+1}%
\]
and the result follows.

\item[(b)] We prove the result for $\frac{A_{n}}{K_{n}}(\beta_{n})=0$ for
$1\leq n\leq N$, the other case being analogous. When $n=1$ we have
$P_{2}(x)=\frac{A_{1}(x)}{K_{1}(x)}\frac{d}{dx}[K_{1}(x)P_{1}(x)]$ and
$K_{1}(x)P_{1}(x)=0$ at $\alpha_{1}$ and $\gamma_{1,1}$ with $\alpha
_{1}<\gamma_{1,1}<$ $\beta_{1}$ and it follows from Rolle's theorem that
$\alpha_{1}<\gamma_{1,2}<\gamma_{1,1}$. The second zero of $P_{2}$ coincides
with $\beta_{1}$, so we have
\[
\alpha_{2}\leq\alpha_{1}<\gamma_{1,2}<\gamma_{1,1}<\gamma_{2,2}=\beta
_{1}<\beta_{2}.
\]
Now, let $2\leq n\leq N$ and assume that we have proved the result for
$P_{n}(x).$ Since $K_{n}(x)P_{n}(x)$ vanishes at $\alpha_{n}$ and
$\gamma_{i,n}$ for $i=1,\dots,n$ with $\alpha_{n}<\gamma_{1,n}<\ldots
<\gamma_{n,n}<\beta_{n}$, Rolle's theorem applied to (\ref{int}) yields
\[
\alpha_{n+1}\leq\alpha_{n}<\gamma_{1,n+1}<\gamma_{1,n}<\dots<\gamma_{n,n}.
\]

Since the largest zero of $P_{n+1}$ coincides with $\beta_{n}$ we have
\[
\alpha_{n+1}\leq\alpha_{n}<\gamma_{1,n+1}<\gamma_{1,n}<\dots<\gamma
_{n,n}<\gamma_{n+1,n+1}=\beta_{n}<\beta_{n+1}%
\]

and the result follows.

\item[(c)] We prove the result for $\frac{A_{n}}{K_{n}}(\beta_{n})=0$ for
$1\leq n\leq N$, the other case being analogous. When $n=1$ we have
$P_{2}(x)=\frac{A_{1}(x)}{K_{1}(x)}\frac{d}{dx}[K_{1}(x)P_{1}(x)]$ and
$K_{1}(x)P_{1}(x)=0$ at $\alpha_{1}$ and $\gamma_{1,1}$ with $\alpha
_{1}<\gamma_{1,1}\leq\beta_{1}$ and it follows from Rolle's theorem that
$\alpha_{1}<\gamma_{1,2}<\gamma_{1,1}$. The second zero of $P_{2}$ coincides
with $\beta_{1}$, so we have
\[
\alpha_{2}\leq\alpha_{1}<\gamma_{1,2}<\gamma_{1,1}\leq\gamma_{2,2}=\beta
_{1}\leq\beta_{2}.
\]
Now, let $2\leq n\leq N$ and assume that we have proved the result for
$P_{n}(x).$ Since $K_{n}(x)P_{n}(x)$ vanishes at $\alpha_{n}$ and
$\gamma_{i,n}$ for $i=1,\dots,n$ with $\alpha_{n}<\gamma_{1,n}<\ldots
<\gamma_{n,n}\leq\beta_{n}$, Rolle's theorem applied to (\ref{int}) yields%
\[
\alpha_{n}<\gamma_{1,n+1}<\gamma_{1,n}<\dots<\gamma_{n-1,n}<\gamma
_{n,n+1}<\gamma_{n,n}.
\]
The largest zero of $P_{n+1}$ coincides with $\beta_{n}$, so we have
\[
\alpha_{n+1}\leq\alpha_{n}<\gamma_{1,n+1}<\gamma_{1,n}<\dots<\gamma_{n,n}%
\leq\gamma_{n+1,n+1}=\beta_{n}\leq\beta_{n+1}%
\]
and the result follows.

\item[(d)] When $n=1$ we have $P_{2}(x)=\frac{A_{1}(x)}{K_{1}(x)}\frac{d}%
{dx}[K_{1}(x)P_{1}(x)]$ and $P_{2}(x)=0$ at $\alpha_{1}$ and $\beta_{1}$ with
$\alpha_{1}=\gamma_{1,2}<\gamma_{2,2}=\beta_{1}$. When $n=2$, $\gamma
_{1,3}=\beta_{2}$ and $\gamma_{3,3}=\beta_{2}$. Furthermore, $P_{3}%
(x)=\frac{A_{2}(x)}{K_{2}(x)}\frac{d}{dx}[K_{2}(x)P_{2}(x)]$ and Rolle's
theorem implies that $\alpha_{1}=\gamma_{1,2}<\gamma_{2,2}=\beta_{1}$. Hence
\[
\alpha_{2}=\gamma_{1,3}<\alpha_{1}=\gamma_{1,2}<\gamma_{2,3}<\gamma
_{2,2}=\beta_{1}<\gamma_{3,3}=\beta_{2}.
\]
Now, let $3\leq n\leq N$ and assume that we have proved the result for
$P_{n}(x).$ The smallest and largest zero of $P_{n+1}$ is $\gamma
_{1,n+1}=\alpha_{n}$ and $\gamma_{n+1,n+1}=\beta_{n}$. The remaining $n-1$
zeros are obtained by applying Rolle's theorem to the function inside the
square brackets in (\ref{int})
\[
\alpha_{n}=\gamma_{1,n+1}<\gamma_{1,n}<\gamma_{2,n+1}<\dots<\gamma
_{n,n+1}<\gamma_{n,n}<\gamma_{n+1,n+1}=\beta_{n}.
\]

\end{itemize}
\end{proof}

\section{Examples}

\medskip We conclude by giving some examples where our results apply,
highlighting particular choices of $A_{n}$ and $B_{n}$ that give rise to known
families of polynomials.

\begin{example}
Let%
\[
A_{n}(x)=\kappa_{n}\left(  1-x^{2}\right)  ,\quad B_{n}(x)=-2\kappa_{n}%
r_{n}x,\quad r_{n}>0,\kappa_{n}\neq0.
\]
Then%
\[
{K_{n}(x)=}\left(  x^{2}-1\right)  ^{r_{n}}%
\]
and Theorem \ref{main} (a) applies. In \cite{MR1471728}, \cite{MR1359933}, the
authors obtain the same result using a different approach.
\end{example}

\begin{example}
The Bell polynomials $\mathfrak{B}_{n}(x)$ are defined by%
\[
\mathfrak{B}_{n}(x)=\sum_{k=0}^{n}S_{k}^{n}x^{k},\quad n=0,1,\ldots,
\]
where $S_{k}^{n}$ is the Stirling number of the second kind. They satisfy the
differential-difference equation
\[
\mathfrak{B}_{n+1}(x)=x\left[  \mathfrak{B}_{n}^{\prime}(x)+\mathfrak{B}%
_{n}(x)\right]  ,
\]
from which we obtain%
\[
A_{n}(x)=x,\quad{K_{n}(x)=e}^{x}.
\]
In this case, we have $\alpha_{n}=-\infty$ and $\beta_{n}=0$ and from Theorem
\ref{main} (c) it follows that the zeros of $\mathfrak{B}_{n}(x)$ are real and
simple and lie in the interval $(-\infty,0].$
\end{example}

\begin{example}
Let $P_{n}(x)$ be the family of polynomials defined by%
\[
P_{n}(x)={}\left(  c\right)  _{n}\,_{2}F_{1}\left(  \left.
\begin{array}
[c]{c}%
-n,\quad b\\
c
\end{array}
\right\vert x\right)  .
\]
In this case, we have
\[
A_{n}(x)=x(1-x),\quad B_{n}(x)=n+c-bx
\]
and thus%
\[
K_{n}(x)=x^{n+c}(x-1)^{b-c-n}.
\]
Provided that $n\in\left(  -c,b-c\right)  $, it follows from Theorem
\ref{main} (a) that the zeros of $P_{n}$ and $P_{n+1}$ are interlacing and lie
in the interval $(0,1).$ The same result is obtained in \cite{MR2267530} using
a different technique.
\end{example}

\begin{acknowledgement}
The work of D. Dominici was supported by a Humboldt Research Fellowship for
experienced researchers provided by the Alexander von Humboldt Foundation. The
work of K. Driver was supported by the National Research Foundation of South
Africa under grant number 2053730. The work of K. Jordaan was supported by the
National Research Foundation of South Africa under grant number 2054423.
\end{acknowledgement}

\noindent\noindent

\begin{thebibliography}{99}                                                                                               %
\bibitem {MR0146097}L.~Carlitz. \newblock Single variable {B}ell polynomials.
\newblock {\em Collect. Math.}, 14:13--25, 1962.

\bibitem {MR2267530}K.~Driver and K.~Jordaan. \newblock Separation theorems
for the zeros of certain hypergeometric polynomials.
\newblock {\em J. Comput. Appl. Math.}, 199(1):48--55, 2007.

\bibitem {MR1471728}F.~Dubeau and J.~Savoie. \newblock On interlacing
properties of the roots of orthogonal and {E}uler-{F}robenius polynomials.
\newblock In \emph{Approximation theory {VIII}, {V}ol.\ 1 ({C}ollege
{S}tation, {TX}, 1995)}, volume~6 of \emph{Ser. Approx. Decompos.}, pages
185--191. World Sci. Publ., River Edge, NJ, 1995.

\bibitem {MR1359933}F.~Dubeau and J.~Savoie. \newblock On the roots of
orthogonal polynomials and {E}uler-{F}robenius polynomials.
\newblock {\em J. Math. Anal. Appl.}, 196(1):84--98, 1995.

\bibitem {frobenius}F.~G. Frobenius. \newblock Uber die {B}ernoullischen
{Z}ahlen und die {E}ulerischen {P}olynome.
\newblock {\em Sitzungsber. Preuss. Akad. Wiss.}, 809--847, 1910.

\bibitem {MR1321452}M.~E. Hoffman. \newblock Derivative polynomials for
tangent and secant. \newblock {\em Amer. Math. Monthly}, 102(1):23--30, 1995.

\bibitem {MR2191786}M.~E.~H. Ismail.
\newblock {\em Classical and quantum orthogonal polynomials in one variable},
volume~98 of \emph{Encyclopedia of Mathematics and its Applications}.
\newblock Cambridge University Press, Cambridge, 2005.

\bibitem {MR2311051}L.~L. Liu and Y.~Wang. \newblock A unified approach to
polynomial sequences with only real zeros.
\newblock {\em Adv. in Appl. Math.}, 38(4):542--560, 2007.

\bibitem {MR0372517}G.~Szeg{\H{o}}. \newblock {\em Orthogonal polynomials}.
\newblock American Mathematical Society, Providence, R.I., fourth edition, 1975.

\bibitem {MR1174072}B.~A. Vertge{\u{\i}}m. \newblock Euler polynomials and
their generalization. \newblock {\em Sibirsk. Mat. Zh.}, 33(2):164--167, 222, 1992.
\end{thebibliography}
\end{document}